\documentclass[a4paper,reqno]{amsart}
\usepackage[utf8]{inputenc}
\usepackage{enumerate}

\allowdisplaybreaks

\newtheorem{theorem}{Theorem}[section]
\newtheorem{lemma}[theorem]{Lemma}
\newtheorem{corollary}[theorem]{Corollary} 
\newtheorem{proposition}[theorem]{Proposition} 

\theoremstyle{definition}

\theoremstyle{remark}
\newtheorem{remark}[theorem]{Remark}

\numberwithin{equation}{section}

\newcommand*{\indbr}[1]{1_{\{#1\}}} 
\DeclareMathOperator{\EE}{\mathbb{E}} 
\newcommand*{\RR}{\mathbb{R}} 
\newcommand*{\eps}{\varepsilon}
\DeclareMathOperator{\Ent}{Ent}	
\DeclareMathOperator{\Med}{Med}	
\DeclareMathOperator{\Var}{Var}	
\DeclareMathOperator{\supp}{Supp} 
\newcommand*{\calTbar}{\overline{\mathcal{T}}}
\newcommand*{\Tbarbf}{\overline{\mathbf{T}}\vphantom{\mathbf{T}}}

\newcommand*{\vphi}{\varphi}

\begin{document}

\title[Convex Log-Sobolev Inequalities]{A characterization of a class of convex Log-Sobolev inequalities on the real line}

\author{Yan Shu}
\address{Universit\'e Paris Ouest Nanterre La D\'efense - Modal'X, 200 avenue de la R\'epublique, 92000 Nanterre, France.}
\email{yshu@u-paris10.fr}
\thanks{}

\author{Micha{\l} Strzelecki}
\address{Institute of Mathematics, University of Warsaw, Banacha 2, 02--097 Warsaw, Poland.}
\email{M.Strzelecki@mimuw.edu.pl}
\thanks{Research partially
supported by the National Science Centre, Poland, grant no. 2015/19/N/ST1/00891 (M.St.).}

\subjclass[2010]{Primary:
60E15. 
Secondary:
26A51, 
26D10. 
}
\date{February 15, 2017.}
\keywords{Concentration of measure, convex functions, log-Sobolev inequality, weak transport-entropy inequalities.}

\begin{abstract}
We give a sufficient and necessary condition for a probability measure $\mu$ on the real line to satisfy the logarithmic Sobolev inequality for convex functions. The condition is expressed in terms of the  unique left-continuous and non-decreasing map transporting the symmetric exponential measure onto $\mu$. The main tool in the proof is the theory  of weak transport costs.

As a consequence,  we obtain dimension-free concentration bounds for the lower and upper tails of convex functions of independent random variables which satisfy the convex log-Sobolev inequality.
\end{abstract}

\maketitle

\section{Introduction and main results}

Let $\mu$ be a Borel probability measure on $\RR^n$. We say that $\mu$ satisfies the (modified) log-Sobolev inequality for a class of functions $\mathcal{F}$ (with cost function $H:\RR^n\to[0,\infty)$ and constant $c<\infty$) if for every $f\in\mathcal{F}$ we have
\begin{equation}\label{eq:mLSI-abstract}
 \Ent_{\mu}(e^f) \leq \int_{\RR^n} H(c\nabla f) e^f d\mu.
\end{equation}
Here $\Ent_{\mu}(g)$ denotes the usual entropy of a non-negative function $g$, i.e.
\begin{equation}\label{eq:mLSI-abstract-1}
 \Ent_{\mu}(g) = \int_{\RR^n} g \ln(g) d\mu - \int_{\RR^n} g d\mu \ln\Bigl(\int_{\RR^n} g d\mu \Bigr)
\end{equation}
if $\int_{\RR^n} g \ln(g) d\mu<\infty$ and $\Ent_{\mu}(g) =\infty$ otherwise.

In the most classical setting where $H(x) = |x|^2$ and $\mathcal{F}$ is the class of $C^1$ functions this inequality was first introduced by Gross \cite{G75}. In this case it can be rewritten in the form
\begin{equation}\label{eq:mLSI-abstract-2}
 \Ent_{\mu}(g^2) \leq 4C \int_{\RR^n} |\nabla g|^2 d\mu,
\end{equation}
where $C= c^2$, or in yet another form which states that the entropy of a non-negative function $g$ is bounded by its Fisher information:
\begin{equation*}
 \Ent_{\mu}(g) \leq C \int_{\RR^n} \frac{|\nabla g|^2}{g} d\mu.
\end{equation*}

Due to its tensorization property the log-Sobolev inequality is a~powerful tool and can be used to obtain dimension-free concentration bounds (via the Herbst argument). It has been investigated also in more general settings of Riemannian manifolds and in context of applications to the study of Markov chains, see e.g. the monographs~\cite{ABC+,BakryGLbook} and the expository article~\cite{diaconis}.

Inequality~\eqref{eq:mLSI-abstract} with $\mathcal{F}=C^1$ and non-quadratic functions $H$ has been studied by Bobkov and Ledoux \cite{BL}, who considered
\begin{equation*}
H(x) = |x|^2 \indbr{|x|\leq\delta} +\infty \indbr{|x|\geq \delta}, \quad x\in\RR^n,
\end{equation*}
and Gentil, Guillin, and Miclo \cite{MR2198019, MR2351133}, with $H$ essentially of the form
\begin{equation*}
\sum_{i=1}^{n}\max\{|x_i|^2, |x_i|^{p} \},\quad  x=(x_1,\ldots,x_n)\in\RR^n, \ p\geq 2.
\end{equation*}
 Also, in the case of modified log-Sobolev inequalities on the real line, Barthe and Roberto \cite{MR2430612}  characterized (under some mild technical conditions) measures satisfying inequality~\eqref{eq:mLSI-abstract} with
\begin{equation*}
H(x) = x^2\indbr{|x|\leq 1} + \frac{\Phi(|x|)}{\Phi(1)} \indbr{|x|>1}, \quad x\in\RR,
\end{equation*}
 where $\Phi:[0,\infty)\to[0,\infty)$ is a sufficiently nice Young function (cf.~\cite{MR2052235}; the quadratic case was handled earlier by Bobkov and G\"otze~\cite{Dual}).

In this paper we are primarily interested in the case when $\mathcal{F}$ is the class of \emph{convex} functions.\footnote{Note that in this case the dimension-free tensorization property still holds, but the alternative formulations~\eqref{eq:mLSI-abstract-1} and~\eqref{eq:mLSI-abstract-2} --- with $g$ being convex, respectively convex and non-negative --- are no longer equivalent to~\eqref{eq:mLSI-abstract}.} The restriction of the class of functions allows us to work with measures which satisfy much weaker regularity conditions. Most importantly, their supports do not need to be connected, which is not the case for many classical functional inequalities such as the Poincar\'e, transport-information or log-Sobolev inequality. On the other hand, a~disturbing issue arises: the log-Sobolev inequality for convex functions yields via standard reasonings  only deviation inequalities for the upper tail of the functions, i.e.
\begin{equation*}
\mu^{\otimes N}\bigl( \{x\in(\RR^n)^N : f(x) \geq \int_{(\RR^n)^N}\! f\  d\mu^{\otimes N}  + t \}\bigr), \quad t\geq 0
\end{equation*}
(in the classical setting of smooth functions one obtains bounds on the lower tail simply by working with $-f$ instead of $f$, but this is precluded in our situation because $-f$ is usually not convex).

For brevity in what follows we shall slightly informally refer to the log-Sobolev inequality for convex functions simply as ``the convex log-Sobolev inequality''.

Our goal is to give an intrinsic characterization of probability measures on the real line for which the convex log-Sobolev inequality holds. As a corollary we will obtain dimension-free concentration bounds for upper \emph{and} lower tails of convex functions of independent random variables satisfying the convex log-Sobolev inequality. Before stating our main result let us recapitulate what is known in the convex setting.

In \cite{MR2163396} Adamczak found a sufficient condition for a probability measure on the real line to satisfy the convex log-Sobolev inequality with $H(x) =x^2$, $x\in\RR$. This has been extended to functions of the form $H(x) = \max\{x^2, x^{(\beta+1)/\beta}\}$, where $\beta\in[0,1]$, by Adamczak and the second named author \cite{MR3456588}.

Recall that in the classical setting there are strong links between the log-Sobolev inequality, transport-entropy inequalities (introduced by Talagrand~\cite{Tal} and subsequently widely studied, see e.g. \cite{Villani} for a complete and detailed introduction) and the infimum convolution inequality (first introduced by Maurey \cite{Mau91}). Specifically, Bobkov and G\"otze \cite{Dual} showed that the transport-entropy inequality is equivalent to some kind of infimum-convolution inequality, and later Bobkov, Gentil, and Ledoux \cite{BGL01} established a condition  equivalent to the log-Sobolev inequality in terms of the infimum-convolution inequality (these results are closely connected with the celebrated Otto-Villani theorem \cite{OV00, MR1846021}).

Similar connections have been observed by Gozlan, Roberto, Samson, and Tetali in~\cite{GRST14}  for convex log-Sobolev inequalities, \emph{weak} transport-entropy inequalities (introduced therein, see Section~\ref{sec:prelim} for definitions and more details), and the convex infimum convolution inequality. Later, Gozlan, Roberto, Samson, Tetali, and the first named author \cite{GRSST15} established a condition sufficient for the convex modified log-Sobolev inequality to hold on the real line. This condition is expressed in terms of the  unique left-continuous and non-decreasing map transporting the symmetric exponential measure onto $\mu$ (more precisely, in terms of the quantity $ \Delta_\mu(h)$ defined below in~\eqref{eq:delta-h-def}; see Section~\ref{sec:tr-U} for a precise statement of the result). In fact, their sufficient condition is weaker than the condition considered in \cite{MR3456588} (and leads to a result formally stronger than the convex log-Sobolev inequality, cf. Proposition~\ref{prop:tr-U-GRSST15}).

On the other hand, it follows from \cite{GRSST15} and the independent work of Feldheim, Marsiglietti, Nayar, and Wang \cite{FMNW15} that in the case when $H$ is quadratic on an interval near zero and then infinite the following are equivalent:
\begin{itemize}
 \item the condition on the tail of the measure $\mu$ from \cite{MR3456588} in the case $\beta=0$,
 \item the condition on monotone transport map obtained in \cite{GRSST15},
 \item the log-Sobolev inequality for convex functions
\end{itemize}
(and further: the convex Poincar\'e inequality, the convex infimum convolution inequality with a quadratic-linear cost function). In what follows we extend this result to more general choices of the function $H$.

In order to formulate our main result we need to introduce some notation. Let $\tau$ be the symmetric exponential measure on $\RR$ with density $\frac{1}{2}e^{-|x|}$. For a Borel probability measure $\mu$ on $\RR$ we denote by $U_\mu$ the unique left-continuous and non-decreasing map transporting $\tau$ onto the reference measure $\mu$. More precisely,
\begin{equation*}
U_{\mu}(x) = F_{\mu}^{-1}\circ F_{\tau} (x) =
\begin{cases}
F_{\mu}^{-1}(\frac{1}{2} e^{-|x|} ) & \text{if } x< 0,\\
F_{\mu}^{-1}(1-\frac{1}{2} e^{-|x|} ) & \text{if } x\geq 0,
\end{cases}
\end{equation*}
where
\begin{equation*}
F_{\mu}^{-1} (t)= \inf \{y\in\RR : F_{\mu}(y) \geq t \} \in\RR\cup\{\pm\infty\}, \quad t\in[0,1],
\end{equation*}
 is the generalized inverse of the cumulative distribution function defined as
 \begin{equation*}
 F_{\mu}(x) = \mu((-\infty, x]), \qquad x\in\RR.
 \end{equation*}
 Denote moreover by
\begin{equation}\label{eq:delta-h-def}
 \Delta_\mu(h) = \sup_{x\in \RR} \{U_{\mu}(x+h) - U_{\mu}(x) \}, \quad h>0,
\end{equation}
the modulus of continuity of $U_{\mu}$ (by all means we can have $\lim_{h\to 0^+} \Delta_\mu(h)>0$).

Our main result is the following. The results are new already in the case of the quadratic function $H(x)=\tfrac{1}{4}x^2$.

\begin{theorem}\label{thm:main}
Let $H:\RR\to[0,\infty)$ be a symmetric convex function, such that $H(x) = \tfrac{1}{4}x^2$ for $x\in[-2t_0,2t_0]$ for some $t_0>0$, and let $H^*:\RR\to[0,\infty)$ be its Fenchel-Legendre transform. Suppose moreover that $\lim_{x\to\infty} H^*(x)/x = \infty$ and that there exist $A\in[1,\infty)$ and $\alpha\in(1,2]$ such that
\begin{equation}\label{eq:H-scaling-intro}
\forall_{x\in\RR} \ \forall_{s\in[0,1]} \quad H(sx) \leq A s^{\alpha} H(x).
\end{equation}
Denote $\theta(t) = H^*(t)$, $t\geq 0$. For a probability measure $\mu$ on the real line the following conditions are equivalent.
\begin{enumerate}[(i)]
   \item
      For every $s>0$ we have $\int_{\RR} e^{s|x|} d\mu(x)<\infty$ and there exists $c>0$ such that
      \begin{equation*}
      \Ent_{\mu}(e^{\varphi}) \leq \int_{\RR} H(c \varphi')e^{\varphi} d\mu
      \end{equation*}
      for every smooth convex Lipschitz function $\varphi:\RR\to\RR$.
   \item
      There exists $b>0$ such that for all $h>0$,
      \begin{equation*}
      \Delta_\mu(h)\leq \frac{1}{b}\theta^{-1}(t_0^2 + h).
      \end{equation*}
\end{enumerate}
\end{theorem}

\begin{remark}\label{rem:dep-const-intro}
The dependence of constants is explicit but complicated and hence we shall only specify it throughout the proof of the theorem. However, in the case when $H(x)=\frac{x^2}{4}$ the dependence of constants can be simplified:  (ii) implies (i) with $c=\frac{1}{\kappa b}$ where $\kappa=\frac{\min(1,t_0)}{210\theta^{-1}(2+t_0^2)}$; (i) implies that $\Delta_\mu(h)\leq 16c(\frac{2}{3}+\sqrt{h/2})$.
\end{remark}

\begin{remark}\label{rem:Hp-intro}
The condition~\eqref{eq:H-scaling-intro} is stable under taking sums or maxima of functions and one can easily check that, for $1<p<\infty$, the function
\begin{equation*}
H(x) = H_p(x)=
  \begin{cases}
    \frac{1}{4}x^2 &\text{if } |x|\leq 2,\\
    \frac{2}{p} (|x/2|^p -1) + 1 &\text{if } |x|> 2,
  \end{cases}
\end{equation*}
satisfies~\eqref{eq:H-scaling-intro} with $\alpha = \min\{p,2\}$ and some $A=A(p)$.
\end{remark}

\begin{remark}\label{rem:equiv-ii}
The condition on $\Delta_\mu$ from (ii) is to some extent related to the conditions considered in~\cite{MR2163396, MR3456588, FMNW15} (see also Lemma~\ref{lem:const-FMNW15} below). Indeed,
 suppose for simplicity that $\mu$ is symmetric and has a nowhere vanishing density. By the definition of $U_\mu$ we have
\begin{equation*}
\mu([U_{\mu}(x+h),\infty)) = \tau([x+h,\infty)) = e^{-h} \tau([x+h,\infty)) = e^{-h}\mu([U_{\mu}(x),\infty))
\end{equation*}
for $x, h\geq 0$. This easy computation shows that (ii) implies:
\begin{enumerate}
   \item[(ii')] There exists $b>0$ such that for every $h>0$,
\begin{equation*}
\mu([x+g(h),\infty))\leq e^{-h} \mu([x,\infty))\ \quad \forall x\geq 0,
\end{equation*}
where $g(h)=\frac{1}{b}\theta^{-1}(t_0^2 + h)$.
\end{enumerate}
The last inequality is similar to the inequalities which appear in the definitions of the classes $\mathcal{M}_{\beta}$ from \cite{MR2163396}, \cite{MR3456588}, and \cite{FMNW15} (but there $h$ is fixed and $x$ is assumed to be sufficiently large; moreover $g=g(x)$ in \cite{MR2163396}, \cite{MR3456588}).
\end{remark}

\begin{remark} In (i) the assumption about exponential integrability is added in order to exclude very heavy-tailed measures for which the only exponentially integrable convex Lipschitz functions are constants and hence the convex log-Sobolev is  trivially satisfied, whereas (ii) cannot hold.
\end{remark}

\begin{remark} In the definition of the log-Sobolev inequality the constant $c$ is introduced as a scaling of the argument of the function $H$ rather than as a multiplicative constant outside of the integral. We decided to use this form because it simplifies some of the calculations in Section~\ref{sec:LSI-tr}. Clearly, in the most common cases, e.g. in the case of the functions $H_p$ from Remark~\ref{rem:Hp-intro}, the two formulations are equivalent up to numerical constants (for the functions $H_p$ those constants depend on $p$).
\end{remark}

Since the condition (ii) is in fact equivalent to an infimum convolution inequality for convex functions \cite{GRSST15} (see also Proposition~\ref{prop:tr-U-GRSST15} and Proposition~\ref{prop:weak-dual} below) we immediately obtain concentration bounds for the upper and lower tails of convex Lipschitz functions. For simplicity we state them in the case $H(x)=\tfrac{1}{4}x^2$, but one can obtain appropriate results also in the case when $H$ is not quadratic. We refer to e.g. \cite[Corollary~3]{FMNW15} for a similar statement (with $H$  quadratic on an interval near zero and than infinite) and to \cite[Corollary 5.11]{GRST14} for a general overview of concentration properties implied by weak transportation inequalities.

\begin{corollary}\label{cor:concentration-x^2}
 Let $\mu$ be a probability measure on $\RR$ such that for every $s>0$ we have $\int_{\RR} e^{s|x|} d\mu(x)<\infty$ and
\begin{equation*}
\Ent_{\mu}(e^{\varphi}) \leq C \int_{\RR} |\varphi'|^2 e^{\varphi} d\mu
\end{equation*}
for every smooth convex Lipschitz function $\varphi:\RR\to\RR$.
Then there exist $A, B <\infty$ (depending only on $C$), such that for any convex (or concave) function $\varphi:\RR^N\to\RR$ which is $1$-Lipschitz (with respect to the Euclidean norm on $\RR^N$) we have
\begin{equation*}
 \mu^{\otimes N}\bigl( \{x\in\RR^N : |\varphi(x) - \Med_{\mu^{\otimes N}}(\varphi)| \geq t \}\bigr) \leq B e^{-t^2/A}, \quad t\geq 0.
\end{equation*}
\end{corollary}

The article is organized as follows. In Section~\ref{sec:prelim} we recall the definitions of the weak transport-entropy inequalities and provide preliminary results about infimum convolution inequalities from~\cite{GRST14} and~\cite{GRSST15}. In Sections~\ref{sec:LSI-tr}  and~\ref{sec:tr-U} we prove that the conditions (i) and (ii) respectively are equivalent to a weak transport-entropy inequality. Finally, in Section~\ref{sec:summary} we summarize the results of the previous sections and give the proof of Theorem~\ref{thm:main} (and Corollary~\ref{cor:concentration-x^2}). We also recapitulate all conditions equivalent to the convex log-Sobolev inequality in the quadratic case and pose some open questions.

\section{A reminder of weak transport-entropy inequalities}
\label{sec:prelim}

In this section, we recall the definition of the weak transport problem and some preliminary results. Although throughout the most part of the paper we work with measures on the real line we introduce the notation in a slightly greater generality. Below $|\cdot|$ denotes the standard Euclidean norm in $\RR^n$.

\subsection{Weak transport costs} For two probability measures $\mu_1$, $\mu_2$ on $\RR^n$ the weak (barycentric) transport cost associated to the convex cost function  $\theta:[0,\infty)\to[0,\infty]$ is defined by the formula
\begin{equation*}
\calTbar_{\theta} (\mu_2 |\mu_1) = \inf_{\pi} \int_{\RR^n} \! \theta\Bigl( \big| x - \int_{\RR^n} \! y \ p(x,dy) \big|\Bigr) \mu_1(dx),
 \end{equation*}
 where the infimum is taken over all couplings $\pi(dxdy) = p(x,dy) \mu_1(dx)$ of $\mu_1$ and $\mu_2$, and $p(x,\cdot)$ denotes the disintegration kernel of $\pi$ with respect to its first marginal. Using probabilistic notation we can write
\begin{equation*}
 \calTbar_{\theta} (\mu_2 |\mu_1) = \inf \EE \theta \bigl(|X_1 - \EE(X_2|X_1)|\bigr),
\end{equation*}
where the infimum is taken over all random variables $X_1\sim \mu_1$, $X_2\sim\mu_2$. The adjective \textit{weak} stands for the fact that, by Jensen's inequality, it is smaller than the classical transport cost,
\begin{equation}\label{eq:calT}
\mathcal{T}_{\theta} (\mu_2, \mu_1) = \inf_{\pi} \int_{\RR^n\times \RR^n} \!\! \theta(| x -  y |) \pi(dx,dy),
 \end{equation}
 considered in the Monge-Kantorovich transport problem. Note also that in contrast to $\mathcal{T}_{\theta}$ the weak transport cost $\calTbar_{\theta}$ is not symmetric.

\subsection{From weak transport inequality to infimum convolution inequality}
Recall that the relative entropy of $\nu$ with respect to $\mu$ is given by the formula
\begin{equation*}
H(\nu|\mu) =\int_{\RR^n} \log\Bigl(\frac{d\nu}{d\mu}\Bigr) d\nu
\end{equation*}
if $\nu$ is absolutely continuous with respect to $\mu$; otherwise we set $H(\nu|\mu)=+\infty$.

We say that a probability measure $\mu$ on $\RR^n$ \emph{satisfies the weak transport-entropy inequality}
 $\Tbarbf^-(\theta)$ (respectively  $\Tbarbf^+(\theta)$) if
 \begin{equation*}
 \calTbar_{\theta}(\mu|\nu) \leq  H(\nu|\mu) \quad \text{(respectively, } \calTbar_{\theta}(\nu|\mu) \leq H(\nu|\mu) \text{)}
 \end{equation*}
 for every probability measure $\nu$ on $\RR^n$ having a finite first moment.
  We say that $\mu$ satisfies $\Tbarbf (\theta)$ if $\mu$ satisfies  both $\Tbarbf^-(\theta)$ and $\Tbarbf^+(\theta)$.

A Bobkov-G\"otze type criterion for the weak-transport inequality was given in \cite[Proposition~4.5]{GRST14} (cf. also~\cite[Lemma~4.1]{GRSST15}). It is formulated in terms of the infimum convolution operator $Q_t^{\theta}$ defined by the formula
\begin{equation*}
Q_t^{\theta} f(x) = \inf_{y\in \RR^n} \big\{ f(y) + t \theta\Bigl(\frac{|x-y|}{t}\Bigr)\big\}, \quad x\in\RR^n,
\end{equation*}
for $t>0$ and $f:\RR^n\to\RR$.

\begin{proposition}[{\cite[Proposition~4.5]{GRST14}}]\label{prop:weak-dual}
Let $\mu$ be a probability measure on $\RR^n$ and let $\theta:[0,\infty)\to[0,\infty]$ be a convex cost function. Then the following holds.
\begin{enumerate}[(a)]
	\item The measure  $\mu$ satisfies $\Tbarbf^-(\theta)$ if and only if
\begin{equation}\label{eq:Tminus-dual}
\int_{\RR^n} \exp(Q_1^{\theta} f) d\mu \exp\Bigl( -\int_{\RR^n} f d\mu\Bigr) \leq 1
\end{equation}
for every convex function $f:\RR^n\to\RR$ bounded from below.
	\item The measure  $\mu$ satisfies $\Tbarbf^+(\theta)$ if and only if
\begin{equation*}
 \exp\Bigl( \int_{\RR^n}Q_1^{\theta} f d\mu\Bigr) \int_{\RR^n} \exp(-f) d\mu \leq 1
\end{equation*}
for every convex function $f:\RR^n\to\RR$ bounded from below.
	\item If the measure  $\mu$ satisfies $\Tbarbf (\theta)$, then
\begin{equation*}
\int_{\RR^n} \exp(Q_t^{\theta} f) d\mu  \int_{\RR^n} \exp(-f)\leq 1
\end{equation*}
for $t=2$ and every convex function $f:\RR^n\to\RR$ bounded from below. Conversely, if the above inequality holds for some $t>0$, then $\mu$ satisfies $\Tbarbf(t\theta(\cdot/t))$.
\end{enumerate}
\end{proposition}

\section{Equivalence of the convex Log-Sobolev inequality\\
and the weak transportation inequality}
\label{sec:LSI-tr}

In this section we establish the equivalence of the convex log-Sobolev inequality and the weak transport-entropy inequality. In the case of the quadratic cost this was done in \cite{GRST14} (see also \cite{MR3456588} for related results for other cost functions). Using the techniques developed therein, in especially the dual formulation \eqref{eq:Tminus-dual}, we extend this result to a wider class of cost functions. We work with measures on the real line, but in contrast to Section~\ref{sec:tr-U} there are no essential problems with extending the results of this section to a higher dimensional setting (cf. \cite{GRST14, MR3456588}).

Let $H:\RR\to[0,\infty)$ be a symmetric convex function, such that $H(x) =\tfrac{1}{4} x^2$ for $x\in[-2t_0, 2t_0]$ for some $t_0>0$. Note that then, the Fenchel-Legendre transform of $H$, given by the formula
\begin{equation*}
H^*(x) = \sup_{y\in\RR} \{ xy - H(y)\}, \quad x\in\RR,
\end{equation*}
is also quadratic near zero (namely, $H^*(x) = x^2$ for $x\in[-t_0,t_0]$, since for such $x$ the supremum in the definition of $H^*$ is attained at $y=2x$). We assume moreover, that there exists $A\in[1,\infty)$ and $\alpha\in(1,2]$ such that
\begin{equation}\label{eq:H-scaling}
\forall_{x\in\RR} \ \forall_{s\in[0,1]} \quad H(sx) \leq A s^{\alpha} H(x).
\end{equation}
Finally, we assume that $\lim_{x\to\infty} H^*(x)/x = \infty$.

The main result of this section is the following.

\begin{proposition}\label{prop:tr-LSI}
For a probability measure $\mu$ on the real line the following conditions are equivalent.
\begin{enumerate}[(i)]
\item There exists $a>0$ such that measure $\mu$ satisfies $\Tbarbf^-(H^*(a\cdot))$.
\item For every $s>0$ we have $\int_{\RR} e^{s|x|} d\mu(x)<\infty$ and there exists $c>0$ such that 
\begin{equation*}
\Ent_{\mu}(e^{\varphi}) \leq \int_{\RR} H(c \varphi')e^{\varphi} d\mu
\end{equation*}
for every smooth convex Lipschitz function $\varphi:\RR\to\RR$.
\end{enumerate}
The dependence of the constants is the following: (i) implies (ii) with $c=2/a$; (ii) implies (i) with $a=  ((\alpha-1)/A)^{1/\alpha} c^{-1}$.
\end{proposition}

The implication (i) $\implies$ (ii) is a general fact and no special assumptions are used in the proof. For the sake of completeness we sketch the main argument here.

\begin{proof}[Proof of Proposition~\ref{prop:tr-LSI}, (i) $\implies$ (ii)]
The exponential integrability follows from the dual formulation of $\Tbarbf^-(H^*(a\cdot))$ tested with the function $x\mapsto s|x|$ (cf.~\cite[p. 86]{MR3456588}).

According to \cite[Proposition~8.3]{GRST14}, (i) implies that the so-called $(\tau)-$log-Sobolev inequality holds: for all
functions $f:\RR\to\RR$ with $\int_{\RR} f e^f d\mu<\infty$ we have
\begin{equation*}
\Ent_\mu(e^f)\leq \frac{1}{1-\lambda}\int_{\RR} (f-R^\lambda f)e^fd\mu, \quad \lambda\in(0,1).
\end{equation*}
Here
\begin{equation*}
R^\lambda f(x):=\inf_p \Bigl\{\int_{\RR} f(y) p(dy)+\lambda H^*\bigl(a|x-\int_{\RR} yp(dy)|\bigr)\Bigr\},
\end{equation*}
where the infimum is taken over all probability measures $p$ on $\RR$ (note that we skip the dependence on $H^*$ in the notation). Since $H^*$ is convex,  for convex functions $f$
the infimum above is achieved at some Dirac measure:
\begin{equation*}
R^\lambda f(x)=\inf_{y \in\RR}\{f(y)+\lambda H^*(a(x-y))\}.
\end{equation*}
Now, by convexity of $f$,
\begin{multline*}
f(x)-R^\lambda f(x)\leq f(x) + \sup_{y\in\RR} \{ -f(y) - \lambda H^*(a(x-y))\} \\
\leq f(x) + \sup_{y\in\RR} \{ -f(x) - f'(x)(y-x) -\lambda H^*(a(x-y))\} = \lambda H\Bigl(\frac{f'(x)}{a\lambda}\Bigr).
\end{multline*}
Thus, after taking $\lambda=1/2$, we arrive at the assertion (with $c=2/a$).
\end{proof}

For the proof of the second implication we need the following two simple lemmas. The first is of independent interest. It is based on an argument of Maurey (cf. \cite[Proof of Theorem~3]{Mau91}), but takes into account the observation that for compactly supported measures it doesn't matter whether for large arguments the cost function is quadratic or equal to $+\infty$. Recall that $|\cdot|$ stands for the Euclidean norm.

\begin{lemma}\label{lem:tech1}
Let $\mu$ be a probability measure on $\RR^n$ such that $x,y\in\supp \mu \implies |x-y|\leq D$  and denote
\begin{equation*}
\theta_D(x) = \begin{cases} \frac{1}{4D^2} |x|^2 & \text{ if } |x|\leq D,\\
+\infty & \text{ if } |x|> D.
\end{cases}
\end{equation*}
Then for any convex function $\varphi:\RR^n\to\RR$ bounded from below
\begin{equation}\label{ineq:inf-con-bounded}
 \int_{\RR^n} e^{Q_1^{\theta_D} \varphi} d\mu \int_{\RR^n} e^{-\varphi} d\mu \leq 1,
\end{equation}
where $Q_1^{\theta_D} \varphi (x) = \inf \{\varphi(y) + \theta_D(x-y) : y\in\RR^n \}$, $x\in\RR^n$, stands for the infimum convolution.

Conversely, if inequality~\eqref{ineq:inf-con-bounded} holds for some $D$, then the support of $\mu$ is bounded: $x,y\in\supp \mu \implies |x-y|\leq D$.
\end{lemma}

\begin{proof}
Assume that the diameter of the support of $\mu$ is bounded by $D$. Take a convex function $\varphi$, bounded from below. By adding a constant to $\varphi$, we may assume that $\inf_{\supp \mu} \varphi =0$. Take any $\eps>0$,  any $x\in\supp\mu$, and let $z\in\supp \mu$ be such that $\varphi(z)<\eps$. Moreover, define $y=(1-\lambda)x+\lambda z$, where $\lambda\in[0,1]$. Then $|x-y|\leq D$ and hence
\begin{align*}
Q_1^{\theta_D} \varphi(x)
&\leq \varphi(y) + \frac{1}{4D^2}|x-y|^2\leq (1-\lambda)\varphi(x) + \lambda\varphi(z) + \frac{\lambda^2}{4D^2}|x-z|^2\\
&\leq (1-\lambda)\varphi(x) + \lambda\eps + \frac{\lambda^2}{4}.
\end{align*}
We now let $\eps\to 0^+$, and then optimize with respect to $\lambda\in[0,1]$: if $\varphi(x)\geq 1/2$ we take $\lambda=1$, and if $0\leq\varphi(x)\leq 1/2$ we take $\lambda=2\varphi(x)$. This gives $Q_1^{\theta_D} \varphi(x)\leq k(\varphi(x))$, where
\begin{equation*}
k(u) = (u-u^2)\cdot1_{\{u\in[0,1/2)\}} + \tfrac{1}{4}\cdot 1_{\{u\geq 1/2\}}.
\end{equation*}
Note that we have $e^{k(u)}\leq 2-e^{-u}$. Indeed, for $u=1/2$ (or larger) the inequality holds, and for $u\in[0,1/2)$ we have
\begin{equation*}
(e^{u-u^2} + e^{-u})/2\leq e^{-u^2/2} \cosh(u-u^2/2)\leq e^{-u^2/2} \cosh(u) \leq 1.
\end{equation*}
Hence
\begin{equation*}
\int e^{Q_1^{\theta_D} \varphi} \mu \leq \int e^{k(\varphi)} d\mu  \leq 2-\int e^{-\varphi} d\mu
\leq \Big(\int e^{-\varphi} d\mu \Big)^{-1}.
\end{equation*}

Conversely, assume that inequality~\eqref{ineq:inf-con-bounded} holds, but there exist $x_0,y_0 \in\supp \mu$ such that
$|x_0-y_0|> D$. Then there exist $\eps, \delta >0$, such that $\mu(B(x_0,\eps))>\delta$ and $\mu(\RR^n \setminus B(x_0,D+\eps))>\delta$.
Consider now $\varphi_a:\RR^n\to \RR$ defined by the formula $\varphi_a(x) = a\operatorname{dist}\bigl(x, B(x_0,\eps)\bigr)$. For $x\in \RR^n\setminus B(x_0,D+\eps)$ we have
\begin{equation*}
Q_1^{\theta_D} \varphi_a (x)
= \inf_{y\in \RR^n : |x-y| \leq D} \bigl\{ a \operatorname{dist}\bigl(y, B(x_0,\eps)\bigr) + \frac{1}{4D^2} |x-y|^2  \bigr\} \geq a\eps.
\end{equation*}
Moreover, $\varphi_a =0$ on $B(x_0,\eps)$. Thus for sufficiently large  $a>0$,
\begin{align*}
\int_{\RR^n} e^{Q_1^{\theta_D} \varphi_a} d\mu \int_{\RR^n} e^{-\varphi_a} d\mu &\geq \int_{\RR^n \setminus B(x_0,D+\eps)} e^{Q_1^{\theta_D} \varphi_a} d\mu \int_{B(x_0,\eps)} e^{-\varphi_a}d\mu\\
&\geq \delta^2\exp(a\eps)>1 ,
\end{align*}
which contradicts inequality~\eqref{ineq:inf-con-bounded}.
\end{proof}

\begin{lemma}\label{lem:tech2}
We have $H(x)\geq A^{-1} t_0^{2-\alpha} x^{\alpha}$ for $x\geq t_0$.
\end{lemma}

\begin{proof} This follows immediately by taking $s = t_0/x$ in condition~\eqref{eq:H-scaling}.\end{proof}

\begin{proof}[Proof of Proposition~\ref{prop:tr-LSI}, (ii) $\implies$ (i)]
Assume that (ii) holds.  Without loss of generality we can assume that $\mu$ is absolutely continuous with respect to the Lebesgue measure. Indeed, if $\gamma$ is a uniform probability distribution on $[0,\delta]$, then by Lemma~\ref{lem:tech1} it satisfies the convex log-Sobolev inequality with a quadratic-linear function
\begin{equation*}
H_0(x) = \delta^2x^2 \indbr{|x|\leq 1/(2\delta)} + (\delta|x|-1/4)\indbr{|x|>1/(2\delta)}
\end{equation*}
(and constant $c=2$). Hence by Lemma~\ref{lem:tech2} the product measure $\mu\otimes \gamma$ on $\RR^2$ satisfies (for sufficiently small $\delta>0$)
\begin{equation}\label{eq:mu-otimes-gamma-LSI}
\Ent_{\mu\otimes \gamma}(e^{\phi}) \leq \int_{\RR^2} \bigl(H(c \phi'_x) + H(c \phi'_y) \bigr) e^{\phi} d\mu\!\otimes\! \gamma
\end{equation}
for all smooth convex Lipschitz functions $\phi:\RR^2\to\RR$.  Let $\vphi:\RR \to \RR$ be a~smooth convex Lipschitz function and let $\phi:\RR^2\to \RR$ be defined by the formula $\phi (x,y) = \varphi(x+\eps y)$, $x,y\in\RR$. Applying~\eqref{eq:mu-otimes-gamma-LSI} to the function $\phi$  and using the assumption~\eqref{eq:H-scaling}, we see  that the convolution $\mu*\gamma_\varepsilon$, where $\gamma_\varepsilon(\cdot) = \gamma(\cdot/\varepsilon)$, satisfies, up to a multiplicative constant which tends to $1$ as $\eps\to 0^+$, the same modified log-Sobolev inequality as $\mu$:
\begin{equation*}
\Ent_{\mu* \gamma_{\eps}}(e^{\vphi}) \leq  (1+A \eps^{\alpha}) \int_{\RR} H(c \vphi') e^{\vphi} d\mu\! *\! \gamma_{\eps}.
\end{equation*}
The reader will easily check that the proof below shows that  $\mu*\gamma_\varepsilon$ satisfies (i)  with $a_\varepsilon  =  ((\alpha-1)/A_\eps)^{1/\alpha} c^{-1}$, where $A_\eps =  A\cdot(1+A \eps^{\alpha})$ (the multiplicative constant $1+A \eps^{\alpha}$ will appear in one place in the estimate of $F'(t)$). By the Lebesgue dominated convergence theorem this implies that $\mu$ satisfies (i) with $a=  ((\alpha-1)/A)^{1/\alpha} c^{-1}$.

Note that if $\mu$ is absolutely continuous, then standard approximation shows that (ii) holds for all convex Lipschitz functions (by the Rademacher theorem the gradient is then almost surely well defined).

Denote for brevity
\begin{equation*}
Q_t f(x) = Q_t^{H^*(\cdot)} f(x) = \inf_{y\in \RR} \big\{ f(y) + t H^*\Bigl(\frac{|x-y|}{t}\Bigr)\big\}
\end{equation*}
and set $F(t) = \int_{\RR} e^{k(t)Q_t\varphi(x)}d\mu(x)$
for $t>0$ (for some function $k$ yet to be determined). Using first the fact that $\partial_t Q_t\varphi + H(\partial_x Q_t\varphi)=0$ almost surely on $(0,\infty)\times\RR$, then the log-Sobolev inequality, and finally the estimate $H(ck(t) \cdot) \leq Ac^{\alpha}k(t)^{\alpha}H(\cdot)$ which follows from the assumption~\eqref{eq:H-scaling} if only $ck(t)\leq 1$, we arrive at
\begin{align*}
k(t)F'(t) &=
k(t) \int_{\RR} e^{k(t)Q_t\varphi(x)} \Bigl( k'(t) Q_t\varphi(x) + k(t)\partial_t Q_t\varphi(x)\Bigr) d\mu(x)\\
&= k(t)\int_{\RR} e^{k(t)Q_t\varphi(x)} \Bigl( k'(t) Q_t\varphi(x) - k(t) H\bigl(\partial_x Q_t\varphi(x)\bigr)\Bigr) d\mu(x)\\
&=
k'(t) F(t)\log F(t) + k'(t)\Ent_{\mu} \bigl(e^{k(t)Q_t\varphi}\bigr)\\
&\quad
 - k^2(t) \int_{\RR} e^{k(t)Q_t\varphi(x)}H\bigl(\partial_x Q_t\varphi(x)\bigr) d\mu(x)\\
&\leq
k'(t) F(t)\log F(t) + k'(t)\int_{\RR} e^{k(t)Q_t\varphi(x)}H\bigl(ck(t)\partial_x Q_t\varphi(x)\bigr) d\mu(x)\\
&\quad
- k^2(t) \int_{\RR} e^{k(t)Q_t\varphi(x)}H\bigl(\partial_x Q_t\varphi(x)\bigr) d\mu(x)\\
&\leq
 k'(t) F(t)\log F(t)\\
 &\quad  + \big[Ac^{\alpha}k'(t)k(t)^{\alpha} - k^2(t) \big]\cdot\int_{\RR} e^{k(t)Q_t\varphi(x)}H\bigl(\partial_x Q_t\varphi(x)\bigr) d\mu(x).
\end{align*}
Denote $\widetilde{A} = A^{1/\alpha}$ and take
\begin{equation*}
k(t) = (\widetilde{A} c)^{-\alpha/(\alpha-1)}\bigl((\alpha-1)t\bigr)^{1/(\alpha-1)}.
\end{equation*}
Then $k(0)=0$, $ck(t)\leq 1$ for $t\in[0,\widetilde{A} ^{\alpha}c/(\alpha-1)]$,
and $A c ^{\alpha}k'(t)k(t)^{\alpha} - k^2(t)=0$.
Thus the above differential inequality is equivalent to $(\log(F(t))/k(t))'\leq 0$ for almost all $t\in(0,\widetilde{A} ^{\alpha}c/(\alpha-1))$, which, since $Q_t\varphi\leq\varphi$, yields
\begin{equation*}
\frac{\log F(t)}{k(t)} \leq  \liminf_{s\to 0^+} \frac{\log F(s)}{k(s)} \leq \lim_{s\to 0^+} \frac{\log( \int_{\RR} e^{k(s)\varphi(x)}d\mu(x))}{k(s)} = \int_{\RR}\varphi d\mu
\end{equation*}
for $t\in(0,\widetilde{A} ^{\alpha}c/(\alpha-1))$. This is exactly the dual formulation of $\Tbarbf^-\bigl(tk(t)H^*(\cdot/t)\bigr)$ (see Proposition~\ref{prop:weak-dual}~(a)). Taking $t=t_*=\widetilde{A}c/(\alpha-1)^{1/\alpha}$ we see that $t_*k(t_*)=1$, $t_*\leq \widetilde{A} ^{\alpha}c/(\alpha-1)$ (recall that $A\geq 1$ and $1<\alpha\leq 2$), and thus also $ck(t_*)\leq 1$. We conclude that $\mu$ satisfies  $\Tbarbf^-\bigl(H^*(a\cdot)\bigr)$ with $a=1/t_* = ((\alpha-1)/A)^{1/\alpha} c^{-1}$.
\end{proof}

\section{Equivalence of the weak transportation inequality\\
and the condition on $U_{\mu}$}
\label{sec:tr-U}
In the previous section we showed  the equivalence of the  convex log-Sobolev inequality and the weak transport-entropy inequality. In this section, working towards the proof of the main theorem, we deal with weak transport-entropy inequalities.

Throughout this section let $\mu$ be a measure on the real line (which is not a Dirac mass) with median $m=F_{\mu}^{-1}(1/2)$. Denote $s_{\mu} = \inf \supp(\mu)\in[-\infty, \infty)$,  $t_{\mu} = \sup \supp(\mu)\in(-\infty, \infty]$.

Let $\theta:[0,\infty)\to[0,\infty)$ be a convex cost function such that $\theta(t)=t^2$ for $t\in[0,t_0]$ for some $t_0>0$. Note that by convexity  $\theta$ is increasing. We moreover assume that $\int_{0}^{\infty} \theta(d_1x +d_2)e^{-x}dx<\infty$ for all $d_1, d_2>0$. Recall that the following result is proved in~\cite{GRSST15}.

\begin{proposition}[{{\cite[Theorem 1.2]{GRSST15}}}] \label{prop:tr-U-GRSST15}
The following conditions are equivalent.
\begin{enumerate}[(i)]
\item There exists $a>0$ such that $\mu$ satisfies $\Tbarbf(\theta(a\cdot))$.
\item There exists $b>0$ such that for all $h>0$ we have
\begin{equation*}
\Delta_\mu(h)\leq \frac{1}{b}\theta^{-1}(t_0^2 + h).
\end{equation*}
\end{enumerate}
The dependence of the constants is the following: (i) implies (ii) with $b=\kappa_1 a$ and (ii) implies (i) with $a=\kappa_2 b$, where $\kappa_1 = \frac{t_0}{8\theta^{-1}(\log(3) +t_0^2)}$, $\kappa_2 = \frac{\min(1,t_0)}{210\theta^{-1}(2+t_0^2)}$.
\end{proposition}

The goal of this section is to provide a proof of the following stronger version of the above proposition, where $\Tbarbf(\theta(a\cdot))$ is replaced by the formally weaker inequality $\Tbarbf^-(\theta(a\cdot))$.

\begin{proposition}\label{prop:tr-U}
The following conditions are equivalent.
\begin{enumerate}[(i)]
\item There exists $a>0$ such that $\mu$ satisfies $\Tbarbf^-(\theta(a\cdot))$.
\item There exists $b>0$ such that for all $h>0$ we have
\begin{equation*}
\Delta_\mu(h) \leq \frac{1}{b}\theta^{-1}(t_0^2 + h).
\end{equation*}
\end{enumerate}
The dependence of the constants is the following: (i) implies (ii) with
\begin{equation*}
b=\frac{\min(a,1)}{16}\Bigl(1+\frac{1}{at_0}\theta^{-1}\Bigl(\frac{\log (2e^{C_\theta/2}-1)}{2}\Bigr)\Bigr)^{-1},
\end{equation*}
where $C_\theta = \int_0^\infty \theta(2+\frac{1}{\log 2}t)e^{-t}dt$, and (ii) implies (i) with $a=\kappa b$, where $\kappa = \frac{\min(1,t_0)}{210\theta^{-1}(2+t_0^2)}$.
\end{proposition}

For the proof we need the following lemma which follows immediately from~\cite[Theorem~2.2]{GozW2} (cf. \cite[Theorem~6.1]{GRSST15}), where the connection between the condition (ii) satisfied by the map $U_\mu$ and transport-entropy inequalities connected to transport cost which are equal to zero in a neighborhood of zero is explained in detail. In what follows the notation $\int_{t_1}^{t_2}$ always denotes an integral over the \emph{open} interval $(t_1,t_2)$.

\begin{lemma}\label{lem1-attempt2}
Let $\beta:[0,\infty)\to[0,\infty)$ be a function which is equal to zero on the interval $[0,t_0]$ and then strictly increasing; denote its inverse by $\beta^{-1}:[0,\infty)\to[t_0,\infty)$. The following conditions are equivalent.
\begin{enumerate}[(i)]
\item There exists $d>0$ such that for all $h>0$ and $x\in\RR$,
\begin{equation*}
\Delta_\mu(h)\leq \tfrac{1}{d}\beta^{-1}(h).
\end{equation*}
\item There exist $k>0$, $K<\infty$ such that
\begin{align*}
\sup_{x\in[m, t_{\mu})}\frac{1}{\mu((x,\infty))}\int_x^{\infty} \exp\bigl(\beta\bigl(k(u-x)\bigr) \bigr) \mu(du)\leq K,\\
\sup_{x\in(s_\mu, m]} \frac{1}{\mu((-\infty,x))}\int_{-\infty}^x \exp\bigl(\beta\bigl(k(x-u)\bigr) \bigr) \mu(du)\leq K.
\end{align*}
\end{enumerate}
The dependence of the constants is the following: (i) implies (ii) with $K=3$ and $k=d\frac{t_0}{18\beta^{-1}(2)}$; (ii) implies (i) with $d = k \frac{t_0}{4\beta^{-1}(\log K)}$.
\end{lemma}

We also need the following preparatory result.

\begin{lemma}\label{lem2-attempt2}
If $\mu$ satisfies $\Tbarbf^-(\theta(a\cdot))$, then
\begin{align*}
\frac{1}{\mu((x,\infty))}\int_x^\infty \theta(a(u-x))\mu(du)\leq C_\theta &\quad  \text{for } \ x\in[m, t_{\mu}),\\
\frac{1}{\mu((-\infty,x))}\int_{-\infty}^{x} \theta(a(x-u))\mu(du)\leq C_\theta &\quad  \text{for }  x\in(s_\mu, m],
\end{align*}
where $C_\theta=\int_0^\infty \theta(2+\frac{1}{\log 2}t)e^{-t}dt$. Moreover, if $\theta (x)=x^2$ then one can choose $C_\theta=1$.
\end{lemma}

\begin{proof}[Proof of Lemma~\ref{lem2-attempt2}] First note that $\mu$ satisfies the convex Poincar\'e inequality:
for any (smooth) convex Lipschitz $f:\RR\to\RR$ we have
\begin{equation}\label{eq:convex-P}
 \Var_{\mu}(f) \leq  \frac{1}{2a^2} \int_{\RR} |f'|^2 d\mu.
\end{equation}
This follows by a standard Taylor expansion argument from the dual formulation of the transport entropy inequality (see Proposition~\ref{prop:weak-dual}~(a)): we plug $\eps f$ instead of $f$, use the estimate
\begin{align*}
Q_1^{\theta} (\eps f)(x)
&= \inf_{y\in\RR} \{ \eps f(x-y) + \theta(y) \} \geq \eps f(x) +\inf_{y\in\RR} \{ -\eps f'(x)y + \theta(y )\} \\
&\geq \eps f(x) - \theta^*(\eps |f'(x)|)
\end{align*}
(valid for convex functions), and take $\eps\to 0^+$ (alternatively, one could use Proposition~\ref{prop:tr-LSI} and deduce the Poincar\'e inequality with a slightly worse constant from the log-Sobolev inequality by a similar argument).

We only need to prove the first inequality (where $x\in[m, t_{\mu})$), the second one can be taken care of in a similar way.

First, we deal with the case $\theta(x):= x^2$. Inequality~\eqref{eq:convex-P} implies that
\begin{equation*}
 A - B \leq \frac{1}{2}\mu((x,\infty)),
\end{equation*}
where $A=\int_x^\infty (a(u-x))^2\mu(du)$ and $B=\left(\int_x^\infty a(u-x)\mu(du)\right)^2$. (This is obtained by testing~\eqref{eq:convex-P} with $u\mapsto \max\{u-x, 0 \}$ --- note that even though $\mu$ can have atoms there are no problems non-differentiability at $u=x$, see \cite[Proposition~4.6]{GozW2}.) By the Cauchy-Schwarz inequality $B\leq A \mu((x,\infty))$ and thus
\begin{equation*}
 A\leq \frac{1}{2}\mu((x,\infty))+ A \mu((x,\infty)),
\end{equation*}
which leads to
\begin{equation*}
 \frac{1}{\mu((x,\infty))}\int_x^\infty \theta(a(u-x))\mu(du)\leq \frac{1}{2(1-\mu((x,\infty)))}\leq 1.
\end{equation*}

Now we turn to the general $\theta$. By the characterization of Bobkov and G{\"o}tze~\cite[Theorem~4.2]{MR1701522} (cf.~\cite[Theorem~1.5]{GRSST15}), there exist $D_1, D_2 > 0$ such that $\Delta_\mu(h)\leq D_1+D_2h$ for all $h\geq0$. Following the proof of~\cite[Theorem~1]{FMNW15} (see Lemma~\ref{lem:const-FMNW15} below for more details), we see that one can choose $D_1=\frac{2}{a}$ and $D_2=\frac{1}{a\log 2}$.

Fix $x\geq m$ and define $v:=\sup\{u : U_{\mu}(u)\leq x\}$. Since the map $U_{\mu}$ is left-continuous we have $U_{\mu}(v) \leq x < U_{\mu}(v+\eps)$. Recall that $\tau$ denotes the exponential measure and that
\begin{equation*}
\int_{\RR} f d\mu =  \int_{\RR} f\circ U_{\mu} d\tau
\end{equation*}
for any measurable function $f$. Thus (note that $\mu((U_{\mu}(v),x]) =0$ if  we have  $U_{\mu}(v) < x$)
\begin{align*}
\frac{1}{\mu((x,\infty))}\int_x^\infty \theta(a(u-x))\mu(du)
&=\frac{1}{\tau((v,\infty))}\int_{v}^\infty \theta\bigl(a(U_{\mu}(u)-x)\bigr)\tau(du)\\
&\leq e^{v}\int_v^\infty \theta\bigl(a(U_{\mu}(u)-U_{\mu}(v))\bigr)e^{-u}du\\
& \leq \int_v^\infty \theta\bigl(a(D_1 +D_2(u-v))\bigr)e^{-(u-v)}du\\
&= \int_0^\infty \theta(a(D_1+D_2t))e^{-t}dt \\
&=\int_0^\infty \theta\bigl(2+\frac{1}{\log 2}t\bigr)e^{-t}dt = C_\theta < \infty,
\end{align*}
where the last inequality follows from the conditions placed on $\theta$.
\end{proof}

Let us clarify here what we mean by ``Following the proof of \cite[Theorem~1]{FMNW15}'' above, especially since in the proof of~\cite[Theorem~1]{FMNW15} the reference measure $\mu$ is assumed to be symmetric and the function $U_{\mu}$ does not appear explicitly.

\begin{lemma}\label{lem:const-FMNW15} We have $U_{\mu}(x+h) - U_{\mu}(x) \leq \frac{4}{a} + \frac{1}{a\log(2)}h $, but the constant $4/a$ may be replaced by $2/a$ if we know that $x$ and $x+h$ are of the same sign. Thus in the above proof one can take $D_1=\frac{2}{a}$ and $D_2=\frac{1}{a\log 2}$.
\end{lemma}

\begin{proof}
Let $X, X'$ be two independent random variables with distribution $\mu$. Fix $u\geq m$ and plug the function $f(x) = \max\{x-u, 0\}$ into~\eqref{eq:convex-P}:
\begin{align*}
 \frac{1}{2a^2}\mu([u, \infty))
 &\geq \Var_{\mu}(f) = \frac{1}{2}\EE (f(X) - f(X'))^2\\
 &\geq \frac{1}{2}\EE (f(X) - f(X'))^2\bigl[ \indbr{X'\leq m}\indbr{X\geq u+2/a} + \indbr{X\leq m}\indbr{X'\geq u+2/a}  \bigr] \\
 &\geq \frac{1}{2} \EE f(X)^2\indbr{X\geq u+2/a} \geq\frac{2}{a^2} \mu([u+2/a,\infty)),
 \end{align*}
Thus,
\begin{equation*}
\mu([u+2/a,\infty))\leq \frac{1}{4}\mu([u,\infty)), \quad u\geq m,
\end{equation*}
and similarly
\begin{equation*}
\mu((-\infty, u-2/a])\leq \frac{1}{4}\mu((-\infty, u]), \quad u\leq m.
\end{equation*}

By the definition of $U_\mu$ (recall that this function is left-continuous) for $x\in \RR$ and all $\eps>0$,
\begin{equation*}
F_\mu(U_\mu(x)-\eps)\leq F_\tau(x)\leq F_\mu(U_\mu(x)).
\end{equation*}
Denote now $h_0= 2\ln(2)$ and let $x, x+h_0\leq 0$. Then, for  $\eps>0$,
\begin{align*}
\mu\bigl((-\infty, U_{\mu}(x+h_0) -2/a-\eps]\bigr)
&\leq \frac{1}{4}\mu\bigl((-\infty, U_{\mu}(x+h_0)-\eps]\bigr)\\
&\leq  e^{-h_0} \tau\bigl((-\infty, x+h_0]\bigr) =  \tau\bigl((-\infty, x]\bigr)\\
&\leq \mu\bigl((-\infty, U_{\mu}(x)]\bigr).
\end{align*}
Hence $U_{\mu}(x+h_0) - U_{\mu}(x) \leq 2/a$ for $x, x+h_0\leq 0$ (since $\eps>0$ was arbitrary). Similarly $U_{\mu}(x+h_0) - U_{\mu}(x) \leq 2/a$ for $x, x+h_0\geq 0$,
since
\begin{align*}
\mu\bigl([ U_{\mu}(x) +2/a +\eps, \infty)\bigr)
&\leq \frac{1}{4}\mu\bigl([U_{\mu}(x)+\eps, \infty)\bigr)
\leq \frac{1}{4}\bigl(1 - F_{\mu}(U_{\mu}(x))\bigr)\\
&\leq e^{-h_0}\bigl(1 - F_{\tau}(x)\bigr)
\leq 1 - F_{\tau}(x+h_0)\\
&\leq \mu\bigl([U_{\mu}(x+h_0)-\eps, \infty)\bigr).
\end{align*}
Using these inequalities in a telescoping manner at most $\lceil h/h_0 \rceil \leq 1 + h/h_0$ times we conclude that
\begin{equation*}
U_{\mu}(x+h) - U_{\mu}(x) \leq \frac{2}{a} + \frac{1}{a\log(2)}h
\end{equation*}
for any $x\in\RR$ and $h\geq 0$ such that $x$ and $x+h$ are of the same sign . If $x<0<x+h$, then the additive constant $2/a$ in the last estimate has to be replaced by $4/a$.
\end{proof}

Now we are ready to prove Proposition~\ref{prop:tr-U}.

\begin{proof}[Proof of Proposition~\ref{prop:tr-U}] Due to Proposition~\ref{prop:tr-U-GRSST15} we only need to check that if $\mu$ satisfies the inequality $\Tbarbf^-(\theta(a\cdot))$, then the condition from (ii) is satisfied by $U_{\mu}$.

Fix $x>m$ and consider the function $f(t)=\theta(a[t-x]_+)$. Then $Q_1^{\theta(a\cdot)}f(t)=0$ if $t\leq x$. For $t>x$,
\begin{align*}
Q_1^{\theta(a\cdot)} f(t)
&=\inf_{y\in\RR} \{\theta(a[y-x]_+)+\theta(a|t-y|)\}\\
&=\inf_{y\in[x,t]} \{\theta(a(y-x))+\theta(a(t-y))\}
= 2\theta(a(t-x)/2),
\end{align*}
where the last equality follows from the fact that we have an inequality due to the convexity of $\theta$ and on the other hand the infimum is attained at $y=(x+t)/2$. Hence the dual formulation of the weak transport inequality (see Proposition~\ref{prop:weak-dual}~(a)) implies that
\begin{equation*}
\mu((-\infty,x])+\int_x^\infty \exp\bigl(2\theta(a(t-x)/2)\bigr)d\mu(t)
\leq \exp\Bigl(\int_x^\infty \theta(a(t-x))\mu(dt)\Bigr).
\end{equation*}

Denote $k=1/2$ and $\beta(u) = 2\theta(a[u-t_0]_+)$ for $u>0$. Since $\theta$ is increasing we have $\beta(k u)\leq 2\theta(au/2)$. Therefore,
\begin{equation}\label{eq2-attempt2}
\int_{x}^{\infty} \exp\bigl(\beta(k(t-x)) \bigr) \mu(dt)
\leq \exp\Bigl(\int_x^\infty \theta(a(t-x))\mu(dt) \Bigr) -1+\mu((x,\infty)).
\end{equation}

By Lemma~\ref{lem2-attempt2} there exists $C_\theta>0$ such that $\int_x^\infty \theta(a(t-x))\mu(dt)\leq C_\theta\mu((x,\infty))$.
Hence, since $\mu((x,\infty))\in[0,1/2]$,
\begin{equation*}
\frac{\int_{x}^{\infty} \exp\bigl(\beta(k(t-x)) \bigr) \mu(dt) }{\mu((x,\infty))}\leq \frac{\exp\bigl(C_\theta\mu((x,\infty))\bigr)-1+\mu((x,\infty))}{\mu((x,\infty))} \leq 2e^{C_\theta/2}-1.
\end{equation*}
One can deal with $x\leq m$ similarly. Since
$$\beta^{-1}(h) = \bigl(t_0+\tfrac{1}{a}\theta^{-1}(h/2)\bigr),$$
 Lemma~\ref{lem1-attempt2} implies that
\begin{align}\label{constant final}
\Delta_\mu(h)
&\leq \tfrac{1}{d}\beta^{-1}(h) =\tfrac{1}{d} \bigl(t_0+\tfrac{1}{a}\theta^{-1}(h/2)\bigr)
\leq \frac{1}{d\min(a,1)}  \bigl(t_0+\theta^{-1}(h/2) \bigr) \nonumber\\
&\leq \frac{2}{d\min(a,1)} \theta^{-1}(t_0^2+h),
\end{align}
where
\begin{equation*}
d= \frac{t_0}{8\beta^{-1}(\log(2e^{C_\theta/2}-1))} =\frac{t_0}{8(t_0+\frac{1}{a}\theta^{-1}(\frac{1}{2}\log(2e^{C_\theta/2}-1)))}
\end{equation*}
(recall that $k=1/2$). This finishes the proof of the implication (i) $\implies$ (ii).
\end{proof}

\section{Summary}
\label{sec:summary}

\subsection{Proof of the main result and dependence of constants for $H(x)=\tfrac{1}{4}x^2$}
The results of the two preceding sections allow us to prove Theorem~\ref{thm:main}. The proof of Corollary~\ref{cor:concentration-x^2} is postponed to the next subsection.

\begin{proof}[Proof of Theorem~\ref{thm:main}] The implication (ii)$\implies$(i) has been proved in \cite{GRSST15}. The implication (i)$\implies$(ii) follows immediately by combining Propositions~\ref{prop:tr-LSI} and~\ref{prop:tr-U}. The only assumption we need to check is that $\int_0^{\infty} \theta(d_1x +d_2)e^{-x}dx<\infty$ for all $d_1, d_2>0$, but this follows from the scaling condition placed on $H$. Indeed, for $s\in(0,1]$,
\begin{equation*}
H^*(y/s) = \bigl( H(s\cdot) \bigr)^*(y) \geq \bigl( A s^{\alpha}H(\cdot) \bigr)^*(y) = As^{\alpha} H^*(y/(A s^{\alpha})).
\end{equation*}
Taking $z\geq 1$ and substituting into the above inequality $s=z^{-1/(\alpha-1)}$ and $y=As =A z^{-1/(\alpha-1)}$ we arrive at $\theta(z) = H^*(z)\leq H^*(A) A^{-1} z^{\alpha/(\alpha-1)}$, which implies the claim.
\end{proof}

As for the dependence of constants, in the case $H(x)=\tfrac{1}{4} x^2$ one can take $A=1$ and $\alpha=2$ in~\eqref{eq:H-scaling-intro}. Let us consider the implication  (i) $\implies$ (ii) from Theorem~\ref{thm:main}.  In Proposition~\ref{prop:tr-LSI} we have $a=1/c$ and moreover we can take $C_{\theta}=1$ in Lemma~\ref{lem2-attempt2}.
Therefore,  inequality~\eqref{constant final} reads
\begin{equation}\label{constantforH}
\Delta_\mu(h)\leq \tfrac{1}{d} \bigl(t_0+c\sqrt{h/2}\bigr) \leq 8\frac{t_0+\frac{2}{3}c}{t_0}\left(t_0+c\sqrt{h/2}\right),
\end{equation}
since
\begin{equation*}
d= \frac{t_0}{8\bigl(t_0+\frac{1}{a}\sqrt{\frac{1}{2}\log(2e^{1/2}-1)})\bigr)}\geq \frac{t_0}{8(t_0+\frac{2}{3}c)}.
\end{equation*}
 Taking $t_0=\frac{2}{3}c$ we obtain the result announced in Remark~\ref{rem:dep-const-intro} (the dependence of constants for the implication (ii) $\implies$ (i)  follows directly from Proposition~\ref{prop:tr-LSI} and Proposition~\ref{prop:tr-U-GRSST15}). In fact, we can take  $t_0=c\sqrt[4]{2h/9}$ (which minimizes the right-hand side of~\eqref{constantforH}) to obtain a slightly better estimate
 \begin{equation*}
 \Delta_\mu(h)\leq 8c\Bigl(2/3+\sqrt{h/2}+2\sqrt[4]{2h/9}\Bigr).
 \end{equation*}

\subsection{Conditions equivalent to the convex log-Sobolev inequality}

Taking into account the results from \cite{GRSST15}, we can state a handful of conditions equivalent to the log-Sobolev inequality on the real line. For simplicity we work with the quadratic cost.

\begin{theorem}\label{thm:summary}
Let $\theta(t) = t^2$ for $t\geq 0$. For a probability measure $\mu$ on the real line the following conditions are equivalent.
\begin{enumerate}[(i)]
\item For every $s>0$ we have $\int_{\RR} e^{s|x|} d\mu(x)< \infty$ and there exists $C>0$ such that
\begin{equation*}
\Ent_{\mu}(e^{\varphi}) \leq C\int_{\RR} |\varphi'|^2e^{\varphi} d\mu
\end{equation*}
 for every smooth convex Lipschitz function $\varphi:\RR\to\RR$.
\item There exists $a, b>0$ such that for all $h>0$,
\begin{equation*}
\Delta_\mu(h)\leq \sqrt{a + bh}.
\end{equation*}
\item There exists $a_1>0$ such that $\mu$ satisfies the inequality $\Tbarbf^-(\theta(a_1\cdot))$.
\item There exists $a_2>0$ such that $\mu$ satisfies the inequality $\Tbarbf(\theta(a_2\cdot))$.
\item There exists $t>0$ such that $\mu$ satisfies the infimum convolution inequality
\begin{equation*}
\int_{\RR} \exp(Q_t^{\theta} f) d\mu  \int_{\RR} \exp(-f)\leq 1
\end{equation*}
for every convex function $f:\RR\to\RR$ bounded from below.
\end{enumerate}
In each of the implications the constants in the conclusion depend
only on the constants in the premise.
\end{theorem}

With this theorem we can easily give the proof of the concentration inequalities from Corollary~\ref{cor:concentration-x^2}.

\begin{proof}[Proof of Corollary~\ref{cor:concentration-x^2}]
By Theorem~\ref{thm:summary} the measure $\mu$ satisfies the inequality $\Tbarbf$ for the quadratic cost. An application of e.g. \cite[Corollary~5.11]{GRST14} completes the proof. Alternatively, for a more self-contained reasoning, one can use item (v) of the above theorem and adapt the approach of \cite{Mau91}.
\end{proof}

\subsection{Relation to Talagrand's inequality}
Let $\theta(t) =t^2$ for $t \geq 0$. We say that a~probability measure $\mu$ on the real line satisfies Talagrand's inequality (with constant $C$) if
\begin{equation*}
\mathcal{T}_{\theta}(\mu, \nu) \leq C H(\nu|\mu)
\end{equation*}
for every probability measure $\nu$, where $\mathcal{T}_{\theta}$  was defined in~\eqref{eq:calT}.
Recall that in the classical setting of smooth functions we have the implication chain
\begin{equation*}
\text{log-Sobolev inequality } \implies \text{ Talagrand's inequality } \implies \text{ Poincar\'e inequality }
\end{equation*}
and these implications are strict (see \cite[Section~4.3]{GozW2} for a nice discussion of counterexamples). From \cite{GRSST15} we also know that Talagrand's inequality is strictly stronger than the convex log-Sobolev inequality. The following corollary explains what additional information is carried by the Talagrand inequality.
It is an immediate consequence of Theorem~\ref{thm:main} and \cite[Theorem~1.1]{GozW2}.

\begin{corollary} A probability measure $\mu$ on the real line satisfies the Talagrand inequality if and only if it satisfies the Poincar\'e inequality for smooth functions and the log-Sobolev inequality for convex functions.
\end{corollary}

\subsection{Open questions}

We conclude with three open questions, which to the best of our knowledge are open even in the case $\theta(t) = t^2$. They address the possibility of extending some of the results of this paper.

Roughly speaking, Theorem~\ref{thm:main} implies that if a measure $\mu$ on the real line satisfies the log-Sobolev inequality for convex functions (i.e. the inequality  $\Tbarbf^-$), then it also satisfies the log-Sobolev inequality for (almost all) concave functions (i.e. the inequality  $\Tbarbf^+$). Does the reverse implication hold true? Which of those implications hold in a higher dimensional space?

In terms of the weak transport-entropy inequalities the above questions read:
\begin{enumerate}
\item[1.] Suppose that a probability measure $\mu$ on $\RR^n$, $n\geq 2$, satisfies the inequality $\Tbarbf^-(\theta(a\cdot))$ for some $a>0$. Does it satisfy the inequality $\Tbarbf(\theta(a' \cdot))$ for some $a'>0$?
\item[2.] Suppose that a probability measure $\mu$ on the real line satisfies the inequality $\Tbarbf^+(\theta(a\cdot))$ for some $a>0$. Does it satisfy the inequality $\Tbarbf^-(\theta(a' \cdot))$, and thus $\Tbarbf(\theta(\min\{a',a\}\cdot ))$, for some $a'>0$?
\end{enumerate}
We refer to \cite[Theorem~8.15]{GRST14} and~\cite[Remark~8.18]{GRST14} for details and subtleties concerning the log-Sobolev inequality for concave functions.

The last question is deliberately somewhat vague. An intuitive way to understand Theorem~\ref{thm:main} is the following: on the real line one can obtain a measure which satisfies the convex log-Sobolev inequality by a local perturbation of another measure which satisfies this inequality. Now take a measure satisfying the convex log-Sobolev inequality in $\RR^n$ and rearrange the mass locally. Does the convex log-Sobolev inequality still hold? More generally, suppose that a probability measure $\mu$ on $\RR^n$, $n\geq 2$, satisfies the convex log-Sobolev inequality. Can one state this fact with a condition which is not expressed in terms of quantifiers ranging over families of functions or measures (like e.g. the conditions in Section~\ref{sec:LSI-tr}) but rather in terms of the measure $\mu$ itself (like the condition (ii) in Theorem~\ref{thm:main})? To put it briefly:
\begin{enumerate}
\item[3.] Find an intrinsic characterization of probability measures on $\RR^n$, $n\geq 2$, which satisfy the convex log-Sobolev inequality.
\end{enumerate}

\section*{Acknowledgments}

The authors thank Rados{\l}aw Adamczak for numerous conversations and many helpful comments, especially for pointing out that the approach of~\cite{Mau91} can be adapted to the setting of Lemma~\ref{lem:tech1}.

\bibliographystyle{amsplain}	
\bibliography{convex-log-sob-arxiv}

\providecommand{\bysame}{\leavevmode\hbox to3em{\hrulefill}\thinspace}
\providecommand{\MR}{\relax\ifhmode\unskip\space\fi MR }
\providecommand{\MRhref}[2]{%
  \href{http://www.ams.org/mathscinet-getitem?mr=#1}{#2}
}
\providecommand{\href}[2]{#2}
\begin{thebibliography}{10}

\bibitem{MR2163396}
Rados{\l}aw Adamczak, \emph{Logarithmic {S}obolev inequalities and
  concentration of measure for convex functions and polynomial chaoses}, Bull.
  Pol. Acad. Sci. Math. \textbf{53} (2005), no.~2, 221--238. \MR{2163396}

\bibitem{MR3456588}
Rados{\l}aw Adamczak and Micha{\l} Strzelecki, \emph{Modified log-{S}obolev
  inequalities for convex functions on the real line. {S}ufficient conditions},
  Studia Math. \textbf{230} (2015), no.~1, 59--93. \MR{3456588}

\bibitem{ABC+}
C{\'e}cile An{\'e}, S{\'e}bastien Blach{\`e}re, Djalil Chafa{\"{\i}}, Pierre
  Foug{\`e}res, Ivan Gentil, Florent Malrieu, Cyril Roberto, and Gr{\'e}gory
  Scheffer, \emph{Sur les in\'egalit\'es de {S}obolev logarithmiques},
  Panoramas et Synth\`eses [Panoramas and Syntheses], vol.~10, Soci\'et\'e
  Math\'ematique de France, Paris, 2000, With a preface by Dominique Bakry and
  Michel Ledoux. \MR{1845806 (2002g:46132)}

\bibitem{BakryGLbook}
Dominique Bakry, Ivan Gentil, and Michel Ledoux, \emph{Analysis and geometry of
  {M}arkov diffusion operators}, Grundlehren der Mathematischen Wissenschaften
  [Fundamental Principles of Mathematical Sciences], vol. 348, Springer, Cham,
  2014. \MR{3155209}

\bibitem{MR2052235}
F.~Barthe and C.~Roberto, \emph{Sobolev inequalities for probability measures
  on the real line}, Studia Math. \textbf{159} (2003), no.~3, 481--497,
  Dedicated to Professor Aleksander Pe\l czy\'nski on the occasion of his 70th
  birthday (Polish). \MR{2052235}

\bibitem{MR2430612}
\bysame, \emph{Modified logarithmic {S}obolev inequalities on {$\Bbb R$}},
  Potential Anal. \textbf{29} (2008), no.~2, 167--193. \MR{2430612}

\bibitem{BL}
S.~Bobkov and M.~Ledoux, \emph{Poincar\'e's inequalities and {T}alagrand's
  concentration phenomenon for the exponential distribution}, Probab. Theory
  Related Fields \textbf{107} (1997), no.~3, 383--400. \MR{1440138}

\bibitem{MR1701522}
S.~G. Bobkov and F.~G\"otze, \emph{Discrete isoperimetric and {P}oincar\'e-type
  inequalities}, Probab. Theory Related Fields \textbf{114} (1999), no.~2,
  245--277. \MR{1701522}

\bibitem{Dual}
S.~G. Bobkov and F.~G{\"o}tze, \emph{Exponential integrability and
  transportation cost related to logarithmic {S}obolev inequalities}, J. Funct.
  Anal. \textbf{163} (1999), no.~1, 1--28. \MR{1682772}

\bibitem{BGL01}
Sergey~G. Bobkov, Ivan Gentil, and Michel Ledoux, \emph{Hypercontractivity of
  {H}amilton-{J}acobi equations}, J. Math. Pures Appl. (9) \textbf{80} (2001),
  no.~7, 669--696. \MR{1846020}

\bibitem{diaconis}
P.~Diaconis and L.~Saloff-Coste, \emph{Logarithmic {S}obolev inequalities for
  finite {M}arkov chains}, Ann. Appl. Probab. \textbf{6} (1996), no.~3,
  695--750. \MR{1410112}

\bibitem{FMNW15}
N.~Feldheim, A.~Marsiglietti, P.~Nayar, and J.~Wang, \emph{A note on the convex
  infimum convolution inequality}, to appear in Bernoulli, preprint (2015),
  {\tt arXiv:1505.00240}.

\bibitem{MR2198019}
Ivan Gentil, Arnaud Guillin, and Laurent Miclo, \emph{Modified logarithmic
  {S}obolev inequalities and transportation inequalities}, Probab. Theory
  Related Fields \textbf{133} (2005), no.~3, 409--436. \MR{2198019}

\bibitem{MR2351133}
\bysame, \emph{Modified logarithmic {S}obolev inequalities in null curvature},
  Rev. Mat. Iberoam. \textbf{23} (2007), no.~1, 235--258. \MR{2351133}

\bibitem{GRSST15}
{N.} Gozlan, {C.} Roberto, {P.M.} Samson, {Y.} Shu, and {P.} Tetali,
  \emph{Characterization of a class of weak transport-entropy inequalities on
  the line}, to appear in Ann. Inst. Henri Poincar{\'e} Probab. Stat., preprint
  (2015), {\tt arXiv:1509.04202v2}.

\bibitem{GRST14}
{N.} Gozlan, {C.} Roberto, {P.M.} Samson, and {P.} Tetali, \emph{Kantorovich
  duality for general transport costs and applications}, to appear in J. Funct.
  Anal., preprint (2014), {\tt arXiv:1412.7480v4}.

\bibitem{GozW2}
Nathael Gozlan, \emph{Transport-entropy inequalities on the line}, Electron. J.
  Probab. \textbf{17} (2012), no. 49, 18. \MR{2946156}

\bibitem{G75}
L.~Gross, \emph{Logarithmic {S}obolev inequalities}, Amer. J. Math. \textbf{97}
  (1975), no.~4, 1061--1083. \MR{0420249 (54 \#8263)}

\bibitem{Mau91}
B.~Maurey, \emph{Some deviation inequalities}, Geom. Funct. Anal. \textbf{1}
  (1991), no.~2, 188--197. \MR{1097258}

\bibitem{OV00}
F.~Otto and C.~Villani, \emph{Generalization of an inequality by {T}alagrand
  and links with the logarithmic {S}obolev inequality}, J. Funct. Anal.
  \textbf{173} (2000), no.~2, 361--400. \MR{1760620}

\bibitem{MR1846021}
Felix Otto and C\'edric Villani, \emph{Comment on: ``{H}ypercontractivity of
  {H}amilton-{J}acobi equations'' [{J}. {M}ath. {P}ures {A}ppl. (9) {\bf 80}
  (2001), no. 7, 669--696; {MR}1846020 (2003b:47073)] by {S}. {G}. {B}obkov,
  {I}. {G}entil and {M}. {L}edoux}, J. Math. Pures Appl. (9) \textbf{80}
  (2001), no.~7, 697--700. \MR{1846021}

\bibitem{Tal}
M.~Talagrand, \emph{Transportation cost for {G}aussian and other product
  measures}, Geom. Funct. Anal. \textbf{6} (1996), no.~3, 587--600.
  \MR{1392331}

\bibitem{Villani}
C\'edric Villani, \emph{Optimal transport}, Grundlehren der Mathematischen
  Wissenschaften [Fundamental Principles of Mathematical Sciences], vol. 338,
  Springer-Verlag, Berlin, 2009, Old and new. \MR{2459454}

\end{thebibliography}

\end{document}